\documentclass[12pt]{article}
\usepackage{amsfonts}
\usepackage{amsmath}
\usepackage{amssymb}
\usepackage{amsthm}
\usepackage{mathrsfs}
\usepackage{graphics}

\newtheorem{theorem}{Theorem}[section]
\newtheorem{lma}[theorem]{Lemma}
\newtheorem{prop}[theorem]{Proposition}
\newtheorem{cor}[theorem]{Corollary}

\theoremstyle{definition}
\newtheorem{defn}[theorem]{Definition}

\newtheorem{rmk}[theorem]{Remark}

\newcommand{\C}{\mathbb{C}}

\newcommand{\id}{\mathord{{\rm id}}}
\newcommand{\im}{\mathop{{\rm Im}}\nolimits}

\def\imm{\mathop{{\rm imm}}\nolimits}

\def\miximm{\mathop{{\rm miximm}}\nolimits}
\def\per{\mathop{{\rm per}}\nolimits}

\title{The norm of the $k$-th derivative of the $\chi$-symmetric power of an operator}

 \author{
S\'onia Carvalho\thanks{Centro de Estruturas Lineares e Combinat\'oria da Universidade de Lisboa, Av Prof Gama Pinto 2, P-1649-003 Lisboa and Departamento de Matem\'atica do ISEL, Rua Conselheiro Em\'\i dio Navarro 1, 1959-007 Lisbon, Portugal (scarvalho@adm.isel.pt).}
\and
Pedro J.\ Freitas\thanks{Centro de Estruturas Lineares e Combinat\'oria, Av Prof Gama Pinto 2, P-1649-003 Lisboa and Departamento de Matem\'atica da Faculdade de Ci\^encias, Campo Grande,
Edif\'\i cio C6, piso 2, P-1749-016 Lisboa. Universidade de Lisboa (pedro@ptmat.fc.ul.pt).}
 }
 \date{April, 2013}
\begin{document}

\maketitle

\begin{abstract}
In this paper we present the exact value for the norm of directional derivatives, of all orders, for symmetric tensor powers of operators on finite dimensional vector spaces. Using this result we obtain an upper bound for the norm of all directional derivatives of immanants. 

This work is inspired in results by R.\ Bhatia, J.\ Dias da Silva, P.\ Grover and T.\ Jain.
\end{abstract} 



\section{Introduction}

Let $V$ and $U$ be finite-dimensional complex vector spaces, and denote by ${\cal L}(V)$ the vector space of linear operators from $V$ to itself. The $k$-th derivative of a map $f:V\to U$ is a multilinear map $D^k f(T)$ from $(\mathcal{L}(V))^k$ to $\mathcal{L}(U)$ defined as 
$$ D^kf(T)(X^1,\ldots ,X^k)= \dfrac{\partial}{\partial t_1\ldots \partial t_k}\Big \vert_{t_1=\ldots =t_k=0}f(T+t_1X^1+\ldots +t_kX^k).$$
The norm of a multilinear operator $\Phi:(\mathcal{L}(V))^k\longrightarrow\mathcal{L}(U) $ is given by $$ \Vert \Phi \Vert = \sup_ { \Vert X^1 \Vert =\ldots =\Vert X^k \Vert =1 } \Vert \Phi (X^1,\ldots ,X^k) \Vert.$$

We will obtain exact values for the norm of $k$-th derivative of the operator $f(T)=K_{\chi}(T)$, where $K_{\chi}(T)$ represents the $\chi$-symmetric tensor power of the operator $T$, that is, the restriction of the operator $\otimes^m T$ to the subspace of $\chi$-symmetric tensors, which we will denote by $V_\chi$.\medskip

Let $\nu_1,\ldots ,\nu_n$ be the singular values of the operator $T$ and denote by $p_t$ the symmetric polynomial of degree $t$ (the number of variables to be specified). In the papers \cite{J}, \cite{G} and \cite{BGJ} the following values were obtained:
$$\Vert D^k \otimes^m T \Vert = \Vert D^k \vee^m T \Vert = \dfrac{m!}{(m-k)!}\Vert T \Vert ^{m-k}=\dfrac{m!}{(m-k)!}\nu_1^{m-k}$$
$$\Vert D^k \wedge^m T \Vert= k!\, p_{m-k}(\nu_1,\ldots ,\nu_m).$$

On the other hand, R.\ Bhatia and J.\ Dias da Silva established in \cite{BDS} a formula for the norm of the first derivative of the $\chi$-symmetric tensor power of an operator. It is again $k!$ times of the value of the elementary symmetric polynomial of degree $m-1$ taken on $m$ singular values of the operator $T$. In all cases, we note that the norm is the value of the elementary symmetric polynomial of degree $m-k$ applied to a certain family of $m$ singular values of $T$ (eventually with repetitions), multiplied by $k!$. 

In this paper we will present a result that generalizes all these cases. We use techniques developed in \cite{BGJ} and \cite{BDS} and a result from \cite{CF}. 

\section{Results on $\chi$-symmetric tensor powers}

We now present some classic facts and notation about $\chi$-symmetric powers that can be found in \cite[chapter 6]{Me}. Let $S_m$ be the symmetric group of degree $m$, $\chi$ an irreducible character of $S_m$ and define $K_\chi \in {\cal L}(\otimes^m V)$ as
$$K_{\chi}=\dfrac{\chi(\id)}{m!}\displaystyle \sum_{\sigma \in S_m} \chi(\sigma)P(\sigma),$$ 
where $\id$ stands for the identity element of $S_m$ and $P(\sigma)(v_1\otimes\ldots \otimes v_m) = v_{\sigma^{-1}(1)}\otimes \ldots \otimes  v_{\sigma^{-1}(m)}$.
The range of $K_{\chi}$ is called the {\em symmetry class of tensors}\/ associated with the irreducible character $\chi$ and it is represented by $V_{\chi}=K_\chi(\otimes^m V)$. We denote
$$v_1\ast v_2 \ast\ldots  \ast v_m=K_{\chi}(v_1 \otimes v_2 \otimes\ldots  \otimes v_m).$$
These vectors belong to $V_{\chi}$ and are called {\em decomposable symmetrised tensors}.

Given $T\in {\cal L}(V)$, it is known that $V_\chi$ is an invariant subspace for $\otimes^m T$. 
We define the {\it $\chi$-symmetric tensor power of $T$} as the restriction of $\otimes^m T$ to $V_\chi$, and denote it by $K_\chi(T)$. \medskip

Let $\Gamma_{m,n}$ be the set of all maps from the set $ \lbrace 1,\ldots ,m\rbrace$ into the set $ \lbrace 1,\ldots ,n\rbrace$. This set can also be identified with the collection of multiindices $\lbrace (i_1,\ldots ,i_m): i_j \leq n \rbrace$. If $\alpha \in\Gamma_{m,n}$, this correspondence associates to $\alpha$ the $m$-tuple $(\alpha(1),\ldots ,\alpha(m))$. In the set $\Gamma_{m,n}$ we will consider the lexicographic order.
The group $S_m$ acts on $\Gamma_{m,n}$ by the action $(\sigma, \alpha) \longrightarrow \alpha \sigma^{-1}$ where $\sigma \in S_m$ and $\alpha \in \Gamma_{m,n}$. The set 
$$ \lbrace \alpha \sigma : \sigma \in S_m \rbrace \subseteq \Gamma_{m,n}$$
is then the orbit of $\alpha$ and the stabilizer of $\alpha$ is the following subgroup of $S_m$:
$$ G_{\alpha}=\lbrace \sigma \in S_m : \alpha \sigma= \alpha \rbrace.$$
Let $\lbrace e_1,\ldots ,e_n \rbrace$ be an orthonormal basis of the vector space $V$. Then $$\lbrace e_{\alpha}^\otimes=e_{\alpha(1)}\otimes e_{\alpha(2)}\otimes\ldots  \otimes e_{\alpha(m)}: \alpha \in \Gamma_{m,n} \rbrace$$ is a basis of the $m$-th tensor power of $V$. So, by the definition of the space $V_\chi$, the set $$\lbrace e^*_{\alpha}=K_\chi( e_{\alpha}^\otimes):  \alpha \in \Gamma_{m,n} \rbrace$$  spans $V_{\chi}$. However, this set need not be a basis of $V_{\chi}$, because its elements might not be linearly independent, some of them may even be zero. Let 
\begin{equation}
\label{omega_chi}
\Omega =\Omega_{\chi}= \lbrace \alpha \in \Gamma_{m,n} : \sum_{\sigma \in G_{\alpha}}\chi(\sigma) \neq 0 \rbrace.
\end{equation}
With simple calculations, we can conclude that 
\begin{equation} 
\Vert e_{\alpha}^{\ast}\Vert^2 = \dfrac{\chi(\id)}{m!}\sum_{\sigma \in G_{\alpha}}\chi(\sigma).
\label{norma-tensores}
\end{equation}
So the nonzero decomposable symmetrised tensors are $\lbrace e_{\alpha}^{\ast}: \alpha \in \Omega \rbrace.$
Now, let $\Delta$ be the system of distinct representatives for the quocient set
$\Gamma_{m,n}/{S_m}$, constructed by choosing the first element in each orbit, for the lexicographic order of indices. It is easy to check that $\Delta \subseteq G_{m,n}$, where $G_{m,n}$ is the set of all increasing sequences of $\Gamma_{m,n}.$ Let $$\overline{\Delta}=\Delta\cap \Omega.$$ It can be proved that the set $\lbrace e_{\alpha}^{\ast}: \alpha \in \overline{\Delta} \rbrace$ is linearly independent. We have already seen that the set $\lbrace e_{\alpha}^{\ast}: \alpha \in \Omega \rbrace,$ spans $V_{\chi}$, so there is a set $\widehat{\Delta}$, such that $\overline{\Delta}\subseteq \widehat{\Delta} \subseteq \Omega$ and 
\begin{equation}
\lbrace e_{\alpha}^{\ast}: \alpha \in \widehat{\Delta} \rbrace,
\end{equation}
is a basis for $V_{\chi}$. It is also known that this basis is orthogonal if $\chi$ is a linear character. If $\chi$ is the principal character, $\widehat \Delta= \overline{\Delta} = G_{m,n}$ and if $\chi$ is the alternating character, $\widehat \Delta = \overline{\Delta}= Q_{m,n}$. In both cases the corresponding bases are orthogonal.\medskip

A partition $\pi$ of $m$ is an $r$-tuple of positive integers $\pi=(\pi_1,\ldots ,\pi_r)$, such that
\begin{itemize}
\item $\pi_1 \geq\ldots  \geq \pi_r,$
\item $\pi_1+\ldots +\pi_r=m.$
\end{itemize}

Sometimes it is useful to consider a partition of $m$ with exactly $m$ entries, so we complete the list with zeros. The number of nonzero entries in the partition $\pi$ is called the {\em length of $\pi$}\/  and is represented by $l(\pi)$.

Given an $n$-tuple of real numbers $x=(x_1,\ldots ,x_n)$ 
and $\alpha\in \Gamma_{m,n}$, we define the $m$-tuple
$$x_\alpha := (x_{\alpha(1)},x_{\alpha(2)},\ldots, x_{\alpha(m)}).$$

It is known from representation theory that there is a canonical correspondence between the irreducible characters of $S_m$ and the partitions of $m$, it is usual to use the same notation to represent both of them.
 Recall that if $\chi=(1,\ldots ,1)$ then $V_{\chi}=\wedge^m V$ is the Grassmann space, and if $\chi=(m,0\ldots ,0)$, then $V_{\chi}=\vee^m V$.
 

For every partition  $\pi=(\pi_1, \pi_2,\ldots ,\pi_{l(\pi)})$ of $m$ we define $\omega(\pi)$ as 
$$\omega(\pi):=(\underbrace{1,\ldots ,1}_{\pi_1 \text{ times}},\underbrace{2,\ldots ,2}_{\pi_2 \text{ times}},\ldots ,\underbrace{l(\pi),\ldots ,l(\pi)}_{\pi_{l(\pi)} \text{ times}})\in G_{m,n} \subseteq \Gamma_{m,n}.$$

For each $\alpha \in \Gamma_{m,n}$ let $\im \alpha=\lbrace i_1,\ldots ,i_l \rbrace$, suppose that $|\alpha^{-1}(i_1)| \geq\ldots  \geq |\alpha^{-1}(i_l)|$. The partition of $m$
\begin{equation}
\label{mu_alpha}
\mu (\alpha):=(|\alpha^{-1}(i_1)| ,\ldots ,|\alpha^{-1}(i_l)| )
\end{equation}
  is called the {\em multiplicity partition}\/ of  $\alpha$.
 
\begin{rmk}
\label{mult_part}
The multiplicity partition of  $\omega(\pi)$ is equal to the partition $\pi$: $\mu(\omega(\pi))=\pi$.
\end{rmk}

We have that $\im \omega(\pi)=\lbrace 1,2, \ldots, l(\pi) \rbrace$ and that $|\alpha^{-1}(i)|= \pi_i$, for every $i=1,2, \ldots l(\pi)$. So 
$$ \mu(\omega(\pi))=(|\alpha^{-1}(1)|,|\alpha^{-1}(2)|,\ldots,|\alpha^{-1}(l(\pi))|)=(\pi_1,\pi_2, \ldots, \pi_{l(\pi)})= \pi.$$

We recall a well known order defined on the set of partitions of $m$. A partition $\mu$ precedes $\lambda$, written $\mu \preceq \lambda$, if for all $ 1 \leq s \leq m$, $$\displaystyle \sum_{j=1}^s \mu_j \leq \displaystyle \sum_{j=1}^s \lambda_j.$$
 
We will also need the following classical result, which is \cite[Th.\ 6.3.7]{Me}.

\begin{theorem}
Let $\chi$ be a partition of  $m$ and  $\alpha \in \Gamma_{m,n}$. Let $\Omega_{\chi}$ and  $\mu(\alpha)$ be as defined in (\ref{omega_chi}) and (\ref{mu_alpha}).
Then $\alpha \in \Omega_{\chi}$ if and only if $\chi$ majorizes $\mu(\alpha)$.
\label{Merris} 
\end{theorem}
 
\section{Formulas on $k$-th derivatives}

We now present a formula for higher order derivatives that generalizes formulas in \cite{BF} and \cite{BGJ}. It is known that, given $X^1,\ldots,X^m\in {\cal L}(V)$, the space $V_\chi$ is invariant for the map defined as
$$X^1\tilde \otimes X^2\tilde \otimes \ldots \tilde \otimes X^m:=\frac{1}{m!} \sum_{\sigma \in S^m} X^{\sigma(1)}\otimes \cdots \otimes X^{\sigma(m)}.$$
See for instance \cite[p.\ 184]{Me}. We will denote by the restriction of this map to $V_\chi$ by
$X^1 * \cdots \ast X^m$ and call it the {\em symmetrized $\chi$-symmetric tensor product} of  the operators $X^1,\ldots,X^m$. We remark that this notation does not convey the fact that the product depends on the character $\chi$. In \cite{BGJ} the following formula is deduced:
\begin{equation}
D^k(\otimes^mT)(X^1,\ldots,X^k) = \frac{m!}{(m-k)!} \underbrace{T\tilde\otimes \cdots \tilde\otimes T}_{m-k \text{ copies}} \tilde\otimes X^1 \tilde\otimes \cdots \tilde\otimes X^k.\label{deriv_tensor_power}
\end{equation}
If $k>m$ all derivatives are zero. From this we can deduce a formula for $D^kK_\chi(T)(X^1,\ldots,X^k)$, using the same techniques.

If $L$ is a linear map, it is known that the derivative of $L$ is $L$ at each point. Then, applying the chain rule, we have that $$D(L\circ f)(a)(x)=L \circ D(f(a)(x)).$$ If the map $f$ is $k$ times differentiable then 
\begin{equation}
D^k(L\circ f)(a)(x^1,\ldots ,x^k)=L\circ D^kf(a)(x^1,\ldots ,x^k).
\label{chain_rule}
\end{equation}

\begin{theorem} Using the notation we have established, we have \label{deriv_k_chi}
$$D^kK_\chi(T)(X^1,\ldots,X^k) = \frac{m!}{(m-k)!} T * \cdots * T * X^1 *\cdots * X^k.$$
If $m=k$ this formula does not depend on $T$, and if $k>m$ all derivatives are zero. 
\end{theorem}
\begin{proof}
Let $Q$ be the inclusion map defined as $Q: V_\chi \longrightarrow \otimes^m V$, so its adjoint operator $Q^*$ is the projection of $\otimes^m V$ onto $V_\chi$. We have
$$T_1 * \cdots * T_m = Q^* (T_1 \tilde\otimes \cdots \tilde\otimes T_m) Q.$$
Both maps $L\mapsto Q^*T$ and $T\mapsto LQ$ are linear, so we can apply formulas \eqref{deriv_tensor_power} and \eqref{chain_rule} and get
\begin{eqnarray*}
D^kK_\chi(T)(X^1,\ldots,X^n) & = & D^k(Q^*(\otimes^m T)Q)(X^1,\ldots,X^k)\\
& = &  Q^* D^k(\otimes^m T)(X^1,\ldots,X^k) Q\\
& = & \frac{m!}{(m-k)!} Q^* (\underbrace{T\tilde\otimes \cdots \tilde\otimes T}_{m-k\text{ times}} \tilde\otimes X^1 \tilde\otimes\cdots \tilde\otimes X^k) Q\\
& = & \frac{m!}{(m-k)!} T * \cdots * T * X^1 *\cdots * X^k. 
\end{eqnarray*}
This concludes the proof.\end{proof}
 
\section{Norm of the $k$-the derivative of $K_\chi(T)$}

We recall that the norm of a multilinear operator $\Phi:(\mathcal{L}(V))^k\longrightarrow\mathcal{L}(U) $  is given by $$ \Vert \Phi \Vert = \sup_ { \Vert X^1 \Vert =\ldots =\Vert X^k \Vert =1 } \Vert \Phi (X^1,\ldots ,X^k) \Vert.$$
The main result of this section is the following theorem. 
 
 \begin{theorem}\label{main_thm}
 Let $V$ be an $n$-dimensional Hilbert space. Let $m$ and $k$ be positive integers such that $1 \leq k \leq m \leq n$, and let $\chi$ be a partition of $m$. Let $T \to K_{\chi}(T)$ be the map that associates to each element of $\mathcal{L}(V)$ the induced operator $K_{\chi}(T)$ on the symmetry class $V_{\chi}.$
 Then the norm of the derivative of order $k$ of this map is given by the formula
 \begin{equation}
 \Vert D^k K_{\chi}(T) \Vert = k!\,  p_{m-k} (\nu_{\omega(\chi)})
 \end{equation}
 where $p_{m-k}$ is the symmetric polynomial of degree $m-k$ in $m$ variables. 
 \end{theorem}
  
The proof of our main result is inspired in the techniques used in \cite{BDS}. We will now highlight the most important features of the proof of our main theorem.

First we will use the polar decomposition of the operator $T$, in the following form: $P=TW$, with $P$ positive semidefinite and $W$ unitary. We will see that $\Vert D^k K_{\chi}(T)\Vert = \Vert D^k K_{\chi}(P)\Vert$. This allows us to replace $T$ by $P$.
After that we observe that the multilinear map $D^k K_{\chi}(P)$ is positive between the two algebras in question, so it is possible to use a multilinear version of the famous Russo-Dye theorem that states that the norm for a positive multilinear map is attained in $(I,I,\ldots ,I)$, where $I$ is the identity operator. This result considerably simplifies the calculations needed to obtain the expression stated in our theorem.

The second part of our proof consists in finding the largest singular value of $D^k K_{\chi}(P)(I,I,\ldots,I)$, which coincides with the value of the norm of $D^kK_\chi(T)$. \medskip


First we will need some properties of the operator $K_\chi$. They all follow from the definitions. 

\pagebreak

\begin{prop}
Let $\chi$ be an irreducible character of $S_m$ and suppose that $S$ and $T$ are in ${\cal L}(V)$ and $v_1, \ldots v_m \in V$. Then
\begin{enumerate}
\item $K_\chi(ST)=K_\chi(S)K_\chi(T)$,
\item $K_\chi(T)(v_1 \ast \cdots \ast v_m)=T(v_1)\ast \cdots \ast T(v_m)$,
\item $K_\chi(T)^\ast=K_\chi(T^\ast)$, where $T^\ast$ is the adjoint operator of $T$,
\item $K_\chi(T)$ is invertible for all invertible $T$ and $K_\chi(T)^{-1}=K_\chi(T^{-1}).$
\end{enumerate}
\end{prop}

By the polar decomposition, we know that for every $T \in {\cal L}(V)$ there are a positive semidefinite operator $P$ and an unitary operator $W$ such that $P=TW$. Moreover, the eigenvalues of $P$ are the singular values of $T$. 

\begin{prop}
With the above notation, we have 
$$\Vert D^k K_{\chi}(T)\Vert = \Vert D^k K_{\chi}(P)\Vert.$$
\end{prop}

\begin{proof}
Let $P=TW$, with $W$ unitary. Then $K_{\chi}(W)$ is also unitary , because
$$\left[ K_{\chi}(W)\right] ^{-1}=K_{\chi}(W^{-1})=K_{\chi}(W^*)=\left[ K_{\chi}(W)\right] ^*.$$
So, we have
\begin{eqnarray*}
&&\Vert D^k K_{\chi}(T)(X^1,\ldots ,X^k) \Vert = \\
&& \qquad = \Vert D^k K_{\chi}(T)(X^1,\ldots ,X^k)K_{\chi}(W) \Vert \\
&& \qquad = \Vert \left(\dfrac{\partial^m}{\partial t_1\ldots \partial t_k}\Big|_{t_1=\ldots =t_k=0} K_{\chi}(T+t_1X^1+\ldots +t_kX^k)\right)K_{\chi}(W) \Vert \\
&& \qquad = \Vert \dfrac{\partial^m}{\partial t_1\ldots \partial t_k}\Big|_{t_1=\ldots =t_k=0} K_{\chi}(T+t_1X^1+\ldots +t_kX^k)K_{\chi}(W)\Vert \\
&& \qquad = \Vert \dfrac{\partial^m}{\partial t_1\ldots \partial t_k}\Big|_{t_1=\ldots =t_k=0} K_{\chi}(P+t_1X^1W+\ldots +t_kX^kW)\Vert \\
&& \qquad = \Vert D^k K_{\chi}(P)(X^1W,\ldots ,X^kW) \Vert
\end{eqnarray*} 
We have $\Vert X^iW \Vert=\Vert X^i \Vert$ and moreover $\left\lbrace XW: \Vert X \Vert =1 \right\rbrace$ is the set of all operators of norm 1, so
\begin{eqnarray*}
\Vert D^k K_{\chi}(T) \Vert &=& \sup_{\Vert X^1\Vert =\ldots =\Vert X^k\Vert=1}\Vert D^k K_{\chi}(T) (X^1,\ldots ,X^k)\Vert \\
&=& \sup_{\Vert X^1\Vert =\ldots =\Vert X^k\Vert=1}\Vert D^k K_{\chi}(P) (X^1W,\ldots ,X^kW)\Vert \\
&=& \Vert D^k K_{\chi}(P) \Vert.
\end{eqnarray*}
This concludes the proof.
\end{proof}

Now we need to estimate the norm of the operator $D^kK_\chi(P)$. For this, we use a result from \cite{BGJ}, a multilinear version of the Russo-Dye theorem, which we quote here. A multilinear operator $\Phi$ is said to be {\em positive}\/ if $\Phi(X^1,\ldots ,X^k)$ is a positive semidefinite operator whenever $X^1,\ldots ,X^k$ are so.

\begin{theorem}[Russo-Dye multilinear version]
Let $\Phi:{\cal L}(V)^k\longrightarrow {\cal L}(U)$ be a positive multilinear operator. Then $$\Vert \Phi \Vert = \Vert \Phi (I,I,\ldots ,I) \Vert.$$
\end{theorem}

We have that $D^k K_{\chi}(P)$ is a positive multilinear operator, since if $X^1,\ldots,X^k$ are positive semidefinite, then by the formula in Theorem \ref{deriv_k_chi}, $D^k(P)(X^1,\ldots,X^k)$ is the restriction of a positive semidefinite operator to an invariant subspace, and thus is positive semidefinite. 

Therefore, 
$$\Vert D^k K_{\chi}(T)\Vert = \Vert D^k K_{\chi}(P)\Vert=\Vert D^k K_{\chi}(P)(I,I,\ldots ,I)\Vert.$$

Now we have to find the maximum eigenvalue of the operator $D^k K_{\chi}(P)(I,I,\ldots ,I)$. This will be done by finding a basis of $V_{\chi}$ formed by eigenvectors for $D^k K_{\chi}(P)(I,I,\ldots ,I)$. If $E=\{e_1,\ldots, e_n\}$ is an orthonormal basis of eigenvectors for $P$, then $\lbrace e^*_{\alpha}: \alpha \in \widehat{\Delta} \rbrace$ will be a basis of eigenvectors for $D^k K_{\chi}(P)(I,I,\ldots ,I)$ (in general, it will not be orthonormal).\medskip

For $\beta\in Q_{m-k,k}$, define $\otimes^m_\beta P$ as the tensor $X^1\otimes\cdots \otimes X^m$, in which $X^i=P$ if $i\in \im \beta$ and $X^i=I$ otherwise. \medskip

\begin{lma} We refer to the notation we have established so far. 
\begin{enumerate}
\item We have 
$$\underbrace{P\tilde \otimes\cdots \tilde \otimes P}_{m-k\text{ times}}\tilde \otimes I \tilde \otimes \cdots \tilde \otimes I= \frac{k!(m-k)!}{m!} \sum_{\beta\in Q_{m-k,k}} \otimes^m_\beta P.$$

\item  Let $v_1,\ldots,v_m$ be eigenvectors for $P$ with eigenvalues $\lambda_1,\ldots,\lambda_m$. Then
$$\sum_{\beta\in Q_{m-k,m}} \otimes^m_\beta P (v_1\otimes\cdots \otimes v_m) = p_{m-k}(\lambda_1,\ldots,\lambda_m)v_1\otimes \ldots \otimes v_m$$
\end{enumerate}
\end{lma}

\begin{proof} 1. It is a matter of carrying out the computations. The factors $k!$ and $(m-k)!$ account for the permutations of $I$ and $P$, respectively, in each summand. We note that the final factor is exactly the inverse of the number of summands, as in the definition of $P\tilde \otimes\cdots \tilde \otimes P\tilde \otimes I \tilde \otimes \cdots \tilde \otimes I$.\bigskip

2. For each $\beta\in Q_{m-k,m}$ we have that $$\otimes^m_\beta P (v_1\otimes\cdots \otimes v_m)=\displaystyle \prod_{i=1}^{m-k} \lambda_{\beta(i)}(v_1\otimes\cdots \otimes v_m).$$
So, 
\begin{eqnarray*}
\sum_{\beta\in Q_{m-k,m}} \otimes^m_\beta P (v_1\otimes\cdots \otimes v_m) &=& \sum_{\beta\in Q_{m-k,m}}\displaystyle \prod_{i=1}^{m-k} \lambda_{\beta(i)}(v_1\otimes\cdots \otimes v_m)\\
 &=&  p_{m-k}(\lambda_1,\ldots,\lambda_m)v_1\otimes \ldots \otimes v_m.
\end{eqnarray*}
This concludes the proof.\end{proof}

The following proposition gives the expression for the eigenvalues of $D^kK_{\chi}(P)(I,I,\ldots ,I)$. 

\begin{prop}\label{eigenvalues_dk}
Let $\alpha \in \widehat{\Delta}$ and define 
$$\lambda(\alpha):=k!\, p_{m-k}(\nu_\alpha).$$
Then $\lambda(\alpha)$ is the eigenvalue of $D^k K_{\chi}(P)(I,I,\ldots ,I)$ associated with the eigenvector   $e^*_{\alpha}$.
\end{prop} 

\begin{proof}
Recall that $E=\{e_1,\ldots, e_n\}$ is an orthonormal basis of eigenvectors for $P$, with eigenvalues $\nu_1,\ldots,\nu_n$. For every $\alpha \in \Gamma_{m,n}$ we have $$e^*_{\alpha}=\dfrac{\chi(\id)}{m!}\sum_{\sigma \in S_m}\chi(\sigma)e^\otimes_{\alpha \sigma}.$$
Then
\begin{eqnarray*}
D^kK_\chi (P)(I,\ldots,I)(e_\alpha^*) & = & 
\frac{m!}{(m-k)!} ( P\tilde \otimes\cdots \tilde \otimes P\tilde \otimes I \tilde \otimes \cdots \tilde \otimes I ) (e_\alpha^*)\\
& = & \frac{m!}{(m-k)!} \frac{k! (m-k)!}{m!} \sum_{\beta\in Q_{m-k,m}}  \otimes_\beta^m P(e_\alpha^*)\\
& = & k! \sum_{\sigma\in S_m}\chi(\sigma)�\sum_{\beta\in Q_{m-k,m}}  \otimes_\beta^m P(e_{\alpha\sigma}^\otimes)\\
& = & k! \sum_{\sigma\in S_m}\chi(\sigma) p_{m-k}(\nu_{\alpha\sigma}) e_{\alpha\sigma}^\otimes\\
& = & k! \sum_{\sigma\in S_m}\chi(\sigma) p_{m-k}(\nu_{\alpha}) e_{\alpha\sigma}^\otimes\\
& = & k! p_{m-k}(\nu_\alpha) e_\alpha^*
\end{eqnarray*}
In the last equations we used the previous lemma and the symmetry of the polynomial $p_{m-k}$. 
So the eigenvalue associated with $ e^*_{\alpha}$ is $\lambda(\alpha)$.\end{proof}

We have obtained the expression for all the eigenvalues of the operator $D^k K_{\chi}(P)(I,\ldots ,I)$, now we have to find the largest one.

\begin{lma}
If $\alpha$ and $\beta$ are in the same orbit, then $\lambda(\alpha)=\lambda(\beta)$.
\end{lma}
\begin{proof} 
If $\alpha$ and $\beta$ are in the same orbit, then there is $\sigma \in S_m$ such that $\alpha \sigma= \beta$. So by the definition of the symmetric elementary polynomials, we have  
$$p_{m-k}(\nu_\beta)=p_{m-k}(\nu_{\alpha\sigma})=p_{m-k}(\nu_\alpha).$$
This concludes the proof.\end{proof} 

According to \cite{Me}, every orbit has a representative in $G_{m,n}$, and this is the first element in each orbit (for the lexicographic order). Therefore, the norm of the $k$-th derivative of $K_{\chi}(T)$ is attained at some $\lambda(\alpha)$ with $\alpha \in \overline{\Delta} \subseteq G_{m,n}$.
We now compare eigenvalues coming from different elements of $\overline\Delta$.

\begin{lma}\label{prec_lambda}
Let $\alpha, \beta $ be elements of $\overline{\Delta}\subseteq G_{m,n}$. Then $\lambda(\alpha)\geq\lambda(\beta)$ if and only if   $\alpha$ precedes $\beta$ in the lexicographic order.
\end{lma}

\begin{proof}
The result follows directly from the expression of the eigenvalues of $D^k K_{\chi}(P)(I,\ldots ,I)$ given in Proposition \ref{eigenvalues_dk}.
\end{proof}

We are now ready to complete the proof of the main theorem. From now on we will also write  $\chi$ to represent also the partition of $m$ associated with the irreducible character $\chi$. 

\begin{proof} (of Theorem \ref{main_thm}). We have that $\omega(\chi) \in \overline{\Delta}$, so we must have $$\Vert D^k K_{\chi}(P)(I,\ldots ,I) \Vert \geq \lambda(\omega(\chi)).$$
Now let $\alpha \in \overline \Delta$. Using the results from Theorem \ref{Merris} and Remark \ref{mult_part}, we have that $\chi = \mu(\omega(\chi))$ majorizes $\mu{(\alpha)}$. By the definition of multiplicity partition, we have that $\omega(\chi)$ precedes $\alpha$ in the lexicographic order. By Lemma \ref{prec_lambda}, we then have $\lambda(\omega(\chi)) \geq \lambda(\alpha)$ and
$$\Vert D^k K_{\chi}(T) \Vert = \Vert D^k K_{\chi}(P)(I,\ldots ,I) \Vert = \lambda(\omega(\chi)) = k!\, p_{m-k}(\nu_{\omega(\chi)}).$$ 
This concludes the proof of the theorem.
\end{proof}

Now we will see that the formulas obtained by Jain \cite{J} and Grover \cite{G} are particular cases of this last formula.
 
If $\chi=(m,0, \cdots,0)$ then $K_\chi(T)=\vee^m T$. In this case $\nu_{\omega(\chi)}=(\nu_1,\ldots,\nu_1).$ So we have 
\begin{eqnarray*}
\Vert D^k \vee^m T \Vert &= &k!p_{m-k}(\nu_{\omega(\chi)})\\
& = & k!p_{m-k}(\nu_1, \nu_1, \ldots, \nu_1)\\
& = & k!\dbinom{m}{k}\nu_1^{m-k}=\dfrac{m!}{(m-k)!}\nu_1^{m-k}\\
& = &\dfrac{m!}{(m-k)!}\Vert T \Vert ^{m-k}
\end{eqnarray*}
   
Also, if $\chi=(1,1,\cdots,1)$, then $K_\chi(T)=\wedge^m T$ and $\nu_{\omega(\chi)}=(\nu_1, \nu_2,\cdots, \nu_m).$ In this case we have that 
$$\Vert D^k \wedge^m T \Vert =  k!\,p_{m-k}(\nu_1,\nu_2,\cdots,\nu_m),$$
where $p_{m-k}(\nu_1,\nu_2,\cdots,\nu_m)$ is the symmetric elementary polynomial of degree $m-k$ calculated on the top $m$ singular values of $T$.\medskip
   
Our main formula also generalizes the result for the norm of the first derivative of $K_\chi(T)$ obtained by R.\ Bhatia and J.\ Dias da Silva in \cite{BDS}. Just notice that if $k=1$, we have that $Q_{1,m}=\lbrace 1,2,\cdots,m \rbrace$, so 
\begin{eqnarray*}
\Vert D K_{\chi}(T) \Vert & = & p_{m-1}(\nu_{\omega(\chi)})\\
& = & \nu_{\omega(\chi)(2)}\nu_{\omega(\chi)(3)}\cdots \nu_{\omega(\chi)(m)}+\nu_{\omega(\chi)(1)}\nu_{\omega(\chi)(3)}\cdots \nu_{\omega(\chi)(m)} + \ldots\\
&& \ldots + \nu_{\omega(\chi)(1)}\nu_{\omega(\chi)(2)}\cdots \nu_{\omega(\chi)(m-1)}\\
& = & \displaystyle \sum_{j=1}^m \displaystyle \prod_{i=1 \atop{i \neq j}}^m \nu_{\omega(\chi)(i)}.
\end{eqnarray*}

\section{Norm of the $k$-th derivative of the immanant}

We now wish to establish an upper bound for the $k$-th derivative of the immanant, which is defined as
$$d_{\chi}(A)=\sum_{\sigma \in S_n} \chi(\sigma)\displaystyle \prod_{i=1}^n a_{i\sigma(i)},$$
where $A$ is a complex $n\times n$ matrix. For this, we recall the definition of $K_{\chi}(A)$, the $m$-th $\chi$-symmetric tensor power of the matrix $A$. We find this definition in \cite{CF}, along with other results we will now use.

A natural way to define this matrix is to fix an orthonormal basis $E$ in $V$, and consider the linear endomorphism $T$ such that $A=M(T,E)$. Define $\mathcal{E'}=(e^*_{\alpha}:\alpha \in \widehat{\Delta})$. 
Let $ \mathcal{E}=(v_\alpha: \alpha \in \widehat \Delta)$ be the orthonormal basis of the $m$-th $\chi$-symmetric tensor power of the vector space $V$ obtained by applying the Gram-Schmidt orthonormalization procedure to $\cal E'$. We define
$$K_\chi(A):=M(K_\chi(T),{\cal E})$$
The matrix $K_\chi(A)$ has rows and columns indexed in $\widehat \Delta$, with  $Q_{m,n}\subseteq \widehat \Delta$. This definition admits, as special cases, the $m$-th compound and the $m$-th induced power of a matrix, as defined in \cite[p.\ 236]{Me}.\medskip 

Since the basis chosen in $V_\chi$ is orthonormal, the result for the norm of the operator applies to this matrix: 
$$\Vert K_\chi(A) \Vert\leq k!\,  p_{m-k} (\nu_{\omega(\chi)}),$$
where $\nu_1 \geq\ldots \geq\nu_n$ are the singular values of $A$. This upper bound is what we will need for the main result in this section. \medskip

Before that, we present an explicit formula for $K_\chi(A)$, from \cite{CF}, and show a similarity between the formulas for $D^kK_\chi(T)$, obtained in Theorem \ref{deriv_k_chi}, and for $D^k K_\chi(A)$, obtained in \cite{CF}.\medskip

Denote by $\imm_{\chi} (A)$ the square matrix with rows and columns indexed by $\widehat \Delta$, whose $(\gamma, \delta)$ entry is $d_{\chi}(A[ \gamma | \delta])$.  Let $B=(b_{\alpha\beta})$ be the change of basis matrix from $\mathcal{E}$ to $\mathcal{E'}$. This means that for each $\alpha\in \widehat \Delta$, 
$$v_\alpha = \sum_{\gamma \in \widehat \Delta} b_{\gamma\alpha} e^*_\gamma.$$
This matrix $B$ does not depend on the choice of the basis $E$ as long as it is orthonormal (it encodes the Gram-Schmidt procedure applied to ${\cal E}'$). 

With these matrices, we can write 
\begin{equation}
K_{\chi}(A) = \frac{\chi(\id)}{m!} B^* \imm_{\overline\chi}(A) B.\label{chi_power_matrix}
\end{equation}

The {\it mixed immanant} of $X^1,\ldots, X^n$ is defined as 
$$\Delta_\chi(X^1,\ldots,X^n):=\frac{1}{n!}\sum_{\sigma\in S_n} d_\chi (X^{\sigma(1)}_{[1]},\ldots, X^{\sigma(n)}_{[n]}),$$
where $X^i_{[j]}$ denotes the $j$-th column of the matrix $X^i$.  If $X^1=\ldots =X^t=A$, for some $t\leq n$ and $A\in M_n(\C)$, we denote the mixed immanant by $\Delta_\chi(A;X^{t+1},\ldots,X^n)$.

Denote by $\miximm_{\chi}(X^1,\ldots,X^n)$ the square matrix with rows and columns indexed by $\widehat \Delta$, whose $(\gamma,\delta)$ entry is $\Delta_{\chi}( X^1[\gamma | \delta],\ldots ,X^n [ \gamma | \delta])$, so that $\miximm_{\chi} (A,\ldots,A)=\imm_{\chi}(A)$. We use the same shorthand as with the mixed immanant: for $k\leq n$, 
$$\miximm_{\chi}(A;X^1,\ldots,X^k):=\miximm_\chi (A,\ldots,A,X^1,\ldots X^k)$$

We have
$$D^k K_{\chi}(A)(X^1,\ldots ,X^k)=\frac{\chi(\id)}{(m-k)!} B^* \miximm_{\overline{\chi}}(A;X^1,\ldots,X^k) B$$

Notice the similarity with the formula in Theorem \ref{deriv_k_chi}. In fact, if $X^i$ is the matrix of the operator $S^i\in {\cal L}(V)$ with respect to the orthonormal basis $\{e_1,\ldots,e_n\}$, then the previous matrix is the matrix of $D^kK_\chi (T)(S^1,\ldots,S^k)$ with respect to the basis $\cal E$.\medskip

We now use the results on the norm in order to get an upper bound for the norm of the $k$-th derivative of the immanant. 

\begin{theorem} Keeping with the notation established, we have that, for $k\leq n$,
$$\Vert D^kd_\chi(A) \Vert \leq k!\, p_{n-k} (\nu_{\omega(\chi)})$$
\end{theorem} 
\begin{proof} We always have $Q_{m,n}\subseteq \widehat \Delta$. We now take $m=n$ and denote $\gamma:= (1,2,\ldots,n)\in Q_{n,n} \subseteq \widehat \Delta$ (this is the only element in $Q_{n,n}$). By definition, $d_\chi(A)$ is the $(\gamma,\gamma)$ entry of $\imm_\chi(A)$, and, according to formula \eqref{chi_power_matrix}, we have
$$\imm_\chi(A)=\frac{n!}{\chi(\id)}(B^*)^{-1}K_{\overline\chi}(A)B^{-1}.$$

Since multiplication by a constant matrix is a linear map, we have
$$D^k ((B^*)^{-1}K_{\overline\chi}(A)B^{-1})(X^1,\ldots,X^k)=(B^*)^{-1}D^k K_{\overline\chi}(A)(X^1,\ldots,X^k)B^{-1}.$$

We denote by $C$ the column $\gamma$ of the matrix $B^{-1}$: 
$$C=(B^{-1})_{[\gamma]}=(b'_{\alpha\gamma}),\quad \alpha\in \widehat\Delta.$$
Then
$$D^kd_\chi(A)(X^1,\ldots, X^k) = \frac{n!}{\chi(\id)} C^* D^k K_{\overline\chi}(A)(X^1,\ldots,X^k) C.$$

By formula \eqref{norma-tensores}, we have that
$$\Vert e_\gamma^* \Vert^2= \frac{\chi(\id)}{n!}.$$

By definition of the matrix $B$, we have 
$$e_\gamma^* = \sum_{\beta \in \widehat \Delta} b'_{\beta\gamma} v_\beta$$
with $C=[b'_{\beta\gamma}:\beta\in \widehat\Delta]$. Since the basis $(v_\alpha:\alpha\in \widehat \Delta)$ is orthonormal, we have
$$\Vert C \Vert^2 = \Vert C \Vert_2^2 = \Vert e_\gamma^* \Vert^2=\frac{\chi(\id)}{n!},$$
where $\Vert C\Vert_2$ is the Euclidean norm of $C$. Therefore,
\begin{eqnarray*}
\Vert D^k d_\chi(A)\Vert & = & \frac{n!}{\chi(\id)}\Vert C D^k K_{\chi}(A) C^* \Vert\\
& \leq & \frac{n!}{\chi(\id)}\Vert C \Vert^2 \Vert D^k K_{\chi}(A) \Vert\\
& = & k!\,  p_{n-k} (\nu_{\omega(\chi)}).
\end{eqnarray*}
This concludes the proof.
\end{proof}

In \cite{BJ} it is proved that this upper bound coincides with the norm of the derivative of the determinant. In \cite{BGJ}, when $d_\chi=\per$, the upper bound presented in formula (52) is, using our notation, $(n!/(n-k)!) \Vert A \Vert^{n-k}$.
Using our formula, we get the same value: for $\omega(\chi) =(1,1,\ldots,1)$, 
$$k!\, p_{n-k} (\nu_{\omega(\chi)})=k!\, \dbinom{n}{n-k} \nu_1^{n-k} = \frac{n!}{(n-k)!} \Vert A\Vert^{n-k}.$$
It is also shown that for
$$A=\begin{pmatrix} 1 & 0 \\ 0 & 0 \end{pmatrix}$$
we have strict inequality, for $d_\chi=\per$.\medskip

One of the purposes of having upper bounds for norms is the possibility of estimating the magnitude of perturbations.  Taylor's formula states that if $f$ is a $p$ times differentiable function between two normed spaces, then
$$f(a+x)-f(a) = \sum_{k=1}^p \frac{1}{k!} D^k f(a) (x,\ldots,x) + O(\Vert x \Vert^{p+1}).$$
Therefore,
$$\Vert f(a+x)-f(a)\Vert \leq \sum_{k=1}^p \frac{1}{k!} \Vert D^k f(a) \Vert \Vert x \Vert^k$$
Using our formulas, we get the following result. 

\begin{cor} According to our notation, we have, for $T,X\in {\cal L}(V)$ and $A,Y\in M_n(\C)$:
$$\Vert K_\chi(T) - K_\chi(T+X) \Vert \leq \sum_{k=1}^m p_{m-k}(\nu_{\omega(\chi)})\Vert X \Vert^k,$$
$$| d_\chi(A) - d_\chi(A+Y) | \leq \sum_{k=1}^n p_{n-k}(\nu_{\omega(\chi)})\Vert Y \Vert^k.$$
\end{cor}

\medskip

{\bf Acknowledgements}. We would like to thank Professors Rajendra Bhatia and Jos\'e Dias da Silva for all the discussions about this subject, both about classical theory and the topics relating to this particular problem. 

\bibliographystyle{plain}

\end{document}